\documentclass[12pt]{amsart}
\usepackage{amsmath,amsthm,amsfonts,amssymb,latexsym,enumerate} 
\usepackage{showlabels}
\usepackage[pagebackref]{hyperref} 

\newcommand{\bC}{{\mathbf C}}

\newcommand{\bF}{{\mathbf F}}

\newcommand{\SSS}{\mathsf{S}}
\newcommand{\AAA}{\mathsf{A}}

\newcommand{\CCC}{\mathsf{C}}
\newcommand{\QQQ}{\mathsf{Q}}

\newcommand{\Aut}{{{\operatorname{Aut}}}}

\newcommand{\Irr}{{{\operatorname{Irr}}}}

\newcommand{\acd}{\operatorname{acd}}
\newcommand{\acs}{\operatorname{acs}}
\newcommand{\sol}{\operatorname{sol}}

\newcommand{\Ker}{\operatorname{Ker}}

\newtheorem{thm}{Theorem}[section]
\newtheorem{lem}[thm]{Lemma}

\newtheorem{que}[thm]{Question}
\newtheorem{exa}[thm]{Example}

\newtheorem*{thmA}{Theorem A}
\newtheorem*{conA'}{Conjecture A'}
\newtheorem*{thmB}{Theorem B}
\newtheorem*{conC}{Conjecture C}

\theoremstyle{definition}

\numberwithin{equation}{section}

\begin{document}

\title[Average character degree]{The average character degree of finite groups and Gluck's conjecture}

\author{Alexander Moret\'o}
\address{Departamento de Matem\'aticas, Universidad de Valencia, 46100
  Burjassot, Valencia, Spain}
\email{alexander.moreto@uv.es}

\thanks{We thank Gabriel Navarro for pointing out an inaccuracy in a previous version of Lemma 2.2 and the anonymous reviewers of an earlier version of this manuscript for many helpful comments.
Research  supported by Ministerio de Ciencia e Innovaci\'on (Grant PID2019-103854GB-I00 funded by MCIN/AEI/ 10.13039/501100011033)  and Generalitat Valenciana AICO/2020/298 and CIAICO/2021/163. }




\begin{abstract}
We prove that the order of a finite group $G$ with trivial solvable radical is bounded above in terms of $\acd(G)$, the average degree of the irreducible characters. 
It is not true that the index of the Fitting subgroup is bounded above in terms of $\acd(G)$, but we show that in certain cases it is bounded in terms of the degrees of the irreducible characters of $G$ that lie over a linear character of the Fitting subgroup.  This leads us to propose a refined version of Gluck's conjecture.
\end{abstract}

\maketitle
  
 \section{Introduction}
 
 In the landmark paper \cite{gr}, R. Guralnick and G. Robinson studied the commuting probability in finite groups. As pointed out in \cite{gr}, if $G$ is a finite group, we may endow $G\times G$  with the structure of a probability space by assigning the uniform distribution.  The probability that a randomly chosen pair of elements of $G$ commute is $k(G)/|G|$, where $k(G)$ is the number of conjugacy classes of $G$ (see \cite{gus}). This number is called the commuting probability of $G$. Note that its inverse $|G|/k(G)$ is the average size of the conjugacy classes of $G$. We will write $\acs(G)=|G|/k(G)$. The main results in \cite{gr} can be reformulated in terms of $\acs(G)$, and we will do so when we refer to results in \cite{gr} in this note.

 Similarly, we write 
 $$
 \acd(G)=\frac{\sum_{\chi\in\Irr(G)}\chi(1)}{k(G)}
 $$
 to denote the average degree of the irreducible characters of $G$. This invariant was introduced in \cite{ilm}, motivated by a conjecture of K. Magaard and H. Tong-Viet \cite{mt}. 
 Since then, $\acd(G)$ has been studied in a number of papers (see, for instance, \cite{mn, qia, lew, ht1, hun, ht2}), but several fundamental questions remain open. The previously done work shows that working with the average character degree tends to be more difficult than working with the average class size. For instance, it was proved in Theorem 11 of \cite{gr} (see also the Addendum for references to P. Lescot papers where this result was proved first), that if $G$ is a finite group and $\acs(G)<12$ then $G$ is solvable. This theorem does not depend on the classification of finite simple groups. The analogous result for the average character degree is that if $\acd(G)<16/5$ then $G$ is solvable. This could only be proved in \cite{mn}, improving on earlier results in \cite{mt} and \cite{ilm},  and uses the classification of finite simple groups. (Note that if $\AAA_5$ is the alternating group on $5$ letters, then $\acs(\AAA_5)=12$ and $\acd(\AAA_5)=16/5$.) 
 In this article, we obtain a version for $\acd(G)$ of  another of the  results in \cite{gr}. We also point out  some results that do not admit an analog for $\acd(G)$, discuss  connections with Gluck's conjecture on the largest character degree and mention some questions that, unfortunately, we have not been  able to solve.  All groups in this paper will be finite. Our notation follows \cite{isa}.
 
By Theorem 8 of \cite{gr} (which does not depend on the classification of finite simple groups), if $G$ is a group with trivial solvable radical then $|G|$ is bounded from above in terms of $\acs(G)$. Using the CFSG, the bound was improved in Theorem 9 of \cite{gr} to $|G|\leq\acs(G)^2$.  The following is our  main result. 
 
 \begin{thmA}
 Let $G$ be a finite group with trivial solvable radical. Then $|G|$ is bounded from above in terms of $\acd(G)$.
 \end{thmA}

 Theorem 9 of \cite{gr} actually asserts that if $G$ is any finite group, then $$|G:\sol(G)|\leq\acs(G)^2,$$ where $\sol(G)$ is the solvable radical of $G$. The proof of this result quickly reduces to the case of groups with trivial solvable radical because, by Lemma 2(ii) of \cite{gr}, $\acs(G/N)\leq \acs(G)$ for any finite group $G$ and $N\trianglelefteq G$. As we discuss in Section 5 the  corresponding question for the average character degree seems more complicated.
 
 It follows from Theorem 4 and Theorem 8 of \cite{gr} that in fact  $|G:\bF(G)|$ is bounded from above in terms of $\acs(G)$, where $\bF(G)$ is the Fitting subgroup of $G$.  Actually, $|G:\bF(G)|\leq\acs(G)^2$ by Theorem 10 of \cite{gr}. It is not difficult to see that $|G:\bF(G)|$ is not bounded in terms of $\acd(G)$. (This is  perhaps surprising in view of the many parallel results between conjugacy class sizes and character degrees.)  If $G$ is the Frobenius group of order $(p-1)p$, where $p$ is any odd prime, then $|G:\bF(G)|=p-1$ is arbitrarily large but $\acd(G)<2$. One could think that, perhaps, this is due to the fact that these groups have arbitrarily many linear characters and just one nonlinear irreducible character. We will see in Example \ref{ex1} that $|G:\bF(G)|$ cannot be bounded from above  in terms of the average of the degrees of the nonlinear irreducible characters either. But what if we just consider certain irreducible characters?

As usual,  given a finite group $G$, $N\trianglelefteq G$, and $\lambda\in\Irr(N)$, we write $\Irr(G|\lambda)$ to denote the set of irreducible characters of $G$ that lie over $\lambda$. We set 
 $$
 \acd(G|\lambda)=\frac{\sum_{\chi\in\Irr(G|\lambda)}\chi(1)}{|\Irr(G|\lambda)|}.
 $$
 With this notation, we have the following result.
 
\begin{thmB}
Let $G$ be a solvable group. Then there exists $\lambda\in\Irr(\bF(G))$ linear  such that $|G:\bF(G)|\leq\acd(G|\lambda)^{\alpha}$, where $\alpha<2.596$.
\end{thmB}

 If $G$ is a finite group, $b(G)$ is the largest degree of the irreducible characters of $G$. Gluck's conjecture \cite{glu} asserts that if $G$ is a solvable group then 
 $|G:\bF(G)|\leq b(G)^2$. D. Gluck \cite{glu} proved that $|G:\bF(G)|\leq b(G)^{13/2}$. There has been a series of improvements on this bound (see, for instance, \cite{mowo, chmn, yan}) but the conjecture remains open. 
  Clearly, $\acd(G|\lambda)\leq b(G)$ and it is not difficult to find  examples where $\acd(G|\lambda)<b(G)$ for every $\lambda\in\Irr(\bF(G))$, so
Theorem B suggests the following refinement of  Gluck's conjecture \cite{glu}.  We remark, however, that our proof just mimics the arguments of known results on Gluck's conjecture. 

\begin{conC}
Let $G$ be a solvable group. Then there exists $\lambda\in\Irr(\bF(G))$ linear such that $|G:\bF(G)|\leq\acd(G|\lambda)^2$.
\end{conC}

We close this Introduction with thanks to N. N.  Hung for many helpful conversations on this paper. In particular, it was he who suggested that, perhaps, the index of the Fitting subgroup of a group  could be bounded in terms of the average degree of the irreducible characters of $G$ lying over some irreducible character of some normal subgroup.

\section{Proof of Theorem A}

In this section, we prove Theorem A.  We start by recalling a well-known consequence of the Krull-Schmidt theorem.

\begin{lem}
\label{aut}
Let $F=N_1\times\cdots\times N_t$ where $N_i$ is the direct product of $u_i\geq1$ copies of a nonabelian simple group $S_i$ for every $i$. Assume that $S_i\not\cong S_j$ if $i\neq j$. Then $\Aut(F)\cong\Aut(N_1)\times\cdots\times\Aut(N_t)$. Furthermore, $\Aut(N_i)\cong\Aut(S_i)\wr\SSS_{u_i}$.
\end{lem}

We need to use cohomology to prove the following extendability criterion.
   
  \begin{lem}
  \label{ext}
  Let $M,N\trianglelefteq G$  with $M\cap N=1$. Put $K=M\times N$, where $M$ and $N$ are   perfect.  Let $\alpha\in\Irr(M)$ and $\beta\in\Irr(N)$ be $G$-invariant. Assume that $\beta$ extends to $G$. Then $\alpha\times\beta$ extends to $G$ if and only if $\alpha$ extends to $G$.
  \end{lem}
  
  \begin{proof}
  Let $\chi\in\Irr(G)$ be an extension of $\beta$. Then $\chi_K$ extends $\beta$ and $\chi(1)=\beta(1)$. Since $M$ is perfect, $\chi_K=1_M\times\beta$, so $\tilde{\beta}=1_M\times\beta$ extends to $G$.   
  
  Suppose first that $\alpha$ extends to $G$. Arguing as before, we deduce that $\tilde{\alpha}=\alpha\times1_N$ extends to $G$. Note that $\alpha\times\beta=\tilde{\alpha}\tilde{\beta}$. 
  By Corollary 6.4 of \cite{imn},  $[\alpha\times\beta]_{G/K}=[\tilde{\alpha}]_{G/K}[\tilde{\beta}]_{G/K}$. By Theorem 11.7 of \cite{isa}, $[\tilde{\beta}]_{G/K}=1=[\tilde{\alpha}]_{G/K}$ so $[\alpha\times\beta]_{G/K}=1.$ By Theorem 11.7 of \cite{isa}, $\alpha\times\beta$ extends to $G$. 
  
  Conversely, assume that $\alpha\times\beta$ extends to $G$. Write $\alpha\times\beta=\tilde{\alpha}\tilde{\beta}$ , where $\tilde{\alpha}=\alpha\times 1_N$ is the unique extension of $\alpha$ to $K$. By Theorem 11.7 of \cite{isa} and Corollary 6.4 of \cite{imn}, 
  $$
  1=[\alpha\times\beta]_{G/K}=[\tilde{\alpha}]_{G/K}[\tilde{\beta}]_{G/K}=[\tilde{\alpha}]_{G/K}.
  $$
  Using Theorem 11.7 of \cite{isa} again, 
  $\tilde{\alpha}$, and hence $\alpha$, extends to $G$, as wanted.
   \end{proof}
   
   This result is false if we remove the hypothesis that $M$ and $N$ are perfect. A counterexample can be found in $G={\tt SmallGroup}(32, 2)$. (I thank G. Navarro for pointing this out.)    
  
  Our next lemma  is an immediate consequence of Jordan's theorem on linear groups.
  
\begin{lem}
\label{jor}
Let $S$ be a nonabelian simple group. Then $|S|$ is bounded from above in terms of $m(S)$, where $m (S)$ is the smallest degree of the nonlinear irreducible characters of $S$. \end{lem}

\begin{proof}
This follows from Jordan's theorem. Using the classification of finite simple groups, sharp bounds were obtained by M. Collins \cite{col}, building on work of Weisfeiler.
\end{proof}

In Theorem A we have the hypothesis that the solvable radical of $G$ is trivial. This implies that the generalized Fitting subgroup of $G$ coincides with the socle of $G$ and is the direct product of the (nonabelian) minimal normal subgroups of $G$ (see 6.5.5 of \cite{ks}, for instance). Recall that if we write $F$ to denote the generalized Fitting subgroup of $G$ then $\bC_G(F)\leq F$ (by Theorem 6.5.8 of \cite{ks}). It follows that if the solvable radical of $G$ is trivial then $G$ is isomorphic to a subgroup of $\Aut(F)$ and by Lemma \ref{aut} we know the structure of this group.  We will use this fact repeatedly.

We need to introduce some notation referring to the irreducible characters of $F$.
Let $$F=S_{11}\times\cdots\times S_{1u_1}\times\cdots\times S_{t1}\times\cdots\times S_{tu_t}$$ be a direct product of nonabelian simple groups with $S_{ij}\cong S_{ik}\cong S_i$ for every $j, k$ and $S_i\not\cong S_j$ for $i\neq j$. Let 
$$
\theta=\alpha_{11}\times\cdots\times\alpha_{1u_1}\times\cdots\times\alpha_{t1}\times\cdots\times\alpha_{tu_t}\in\Irr(F).
$$ 
We will say that each of the characters $\alpha_{ij}$ is a factor of $\theta$.  Each factor $\alpha_{ij}$ of $\theta$ is an irreducible character of a simple group $S_{ij}$. Given a factor $\alpha$ of $\theta$, we will write $F_{\alpha}\leq F$ to denote the direct product of the direct factors of $F$ whose factor in $\theta$ is $\alpha$. Note that
$$F=\prod F_{\alpha}\text{\,\,\,\,\,\,\,(direct product)}$$
where the product runs over the factors $\alpha$ of $\theta$.  Write $\Gamma=\Aut(F)$. Notice that for every factor $\alpha$ of $\theta$, $F_{\alpha}$ is normal in $I_{\Gamma}(\theta)$.  We will refer to the decomposition $F=\prod F_{\alpha}$ as the $\theta$-decomposition of $F$.

For every simple group $S_i$, we fix $\alpha_i\in\Irr(S_i)$ such that $\alpha_i$  extends to $\Aut(S_i)$ (such a character exists by Lemma 4.2 of \cite{mor}). 
If $\theta$ has ``many" principal factors associated to any given simple group $S_i$, we define a character $\theta'\in\Irr(F)$ by means of the following rules:

{\bf Step 1:}
For every $i$, if the number of factors $\alpha_{ij}$ of $\theta$ that are $1_{S_i}$ is bigger than $u_i/2$, then we replace $1_{S_i}$ by  $\alpha_i$  all the times. We write $\overline{\theta}$ to denote this new character. 

{\bf Step 2:}
 In this case, if $\alpha_i$ was already a factor of $\theta$, then we replace the $\alpha_i$ factors of $\theta$ by $1_{S_i}$. We write $\theta'$ to denote this new character. 
 
 Note that the $\theta$-decomposition, the $\overline{\theta}$-decomposition and the $\theta'$-decompositions of $F$ coincide.

If the number of factors $\alpha_{ij}$ in $\theta$ that are $1_{S_i}$ is at most $u_i/2$ for every $i$, we do not define $\theta'$ and we say that $\theta'$ is not defined. Note that if $\theta'$ is not defined then at least one-half of the factors corresponding to the copies of each $S_i$ is nonprincipal so
$$
\theta(1)\geq \prod_{i=1}^t m(S_i)^{u_i/2}.
$$

In the next result we collect some key properties of $\theta$ and $\theta'$ when $\theta'$ is defined.

\begin{lem}
\label{mam}
Let $G$ be a group with trivial solvable radical. 
Let $F$ as before be the generalized Fitting subgroup of $G$. Let $\theta\in\Irr(F)$ and assume that $\theta'$ is defined.
The characters $\theta$ and  $\theta'$ satisfy the following properties:
\begin{enumerate}
\item
$\theta'(1)\geq \prod_{i=1}^t m(S_i)^{u_i/2}$. 
\item
$I_G(\theta)=I_G(\theta')$.
\item
$|\Irr(G|\theta)|=|\Irr(G|\theta')|$
\item
If $\gamma\in\Irr(F)$ is such that $\gamma'$ is defined (possibly $\gamma=\theta$), then $\theta$ and $\gamma'$ are not $G$-conjugate. Furthermore, if $\theta$ and $\gamma$ are not $G$-conjugate, then $\theta'$ and $\gamma'$ are not $G$-conjugate.
\end{enumerate}
\end{lem}

\begin{proof}
(i) It suffices to note that $\theta'$ has been defined so that the number of nonprincipal irreducible factors corresponding to the copies of each $S_i$ is at least one-half of the number of copies of $S_i$. 

(ii) Let $\Gamma=\Aut(F)$. It follows from Lemma \ref{aut} that $I_{\Gamma}(\theta)=I_{\Gamma}(\theta')$. Since $G$ is a subgroup of $\Gamma$, the result follows. 

(iii) Put $T=I_G(\theta)=I_G(\theta')$. By Clifford's correspondence (Theorem 6.11 of \cite{isa}), it suffices to see that $|\Irr(T|\theta)|=|\Irr(T|\theta')|$. These numbers are known to be the number of $\theta$-special classes and the number of $\theta'$-special classes, respectively (Problem 11.10 of \cite{isa}).  By Problem 11.9 of \cite{isa}, it suffices to see that if $H\leq T$, then $\theta$ extends to $H$ if and only if $\theta'$ extends to $H$. 

Let $H\leq T$ and assume that $\theta$ extends to $H$. Our task is to see that $\theta'$ extends to $H$. Put $N=\Ker\theta$. Notice that $N$ is the direct product  of the simple groups $S_{ij}$ whose corresponding factor is the principal character. Let $M$ be the direct product of the remaining direct factors of $F$, so that $F=M\times N$. Note that since $\theta$ is $T$-invariant $M$ and $N$ are normal in $T$. In particular,  $M$ and $N$ are normal in $H$ too.
Write $\theta=\tilde{\theta}\times 1_N$. By Lemma \ref{ext}, $\tilde{\theta}$ extends to $H$ and by Lemma \ref{ext} again, so does the character obtained after the first step in the transformation from $\theta$ to $\theta'$ (note that $\overline{\theta}=\tilde{\theta}\times\varphi$ for some character $\varphi$ all of whose factors are either principal characters or $\alpha_i$ for some $i$). If we need to replace some factors of this character $\overline{\theta}$ by principal characters, then we can apply Lemma \ref{ext} for a third time to deduce that $\theta'$ extends to $H$. 
	
	Analogously, one can see that if $\theta'$ extends to $H$, then $\theta$ extends to $H$ too. This completes the proof.

(iv) Again, it follows from Lemma \ref{aut} and the fact that $G$ is a subgroup of $\Aut(F)$. 
\end{proof}

 If $N$ is a normal subgroup of a group $G$ and $\mathcal{T}\subseteq\Irr(N)$, we write 
$$\Irr(G|\mathcal{T})=\bigcup_{\varphi\in\mathcal{T}}\Irr(G|\varphi).$$
 Also, we write 
$$\acd(G|\mathcal{T})=\frac{\sum_{\chi\in\Irr(G|\mathcal{T})}\chi(1)}{|\Irr(G|\mathcal{T})|}$$
for the average of the degrees of the irreducible characters of $G$ that lie over characters in $\mathcal{T}$. Now, we complete the proof of Theorem A.

\begin{thm}
Let $G$ be a group with trivial solvable radical. Then $|G|$ is bounded (from above) in terms of $\acd(G)$. 
\end{thm}

\begin{proof}
Let $F$, as before, be the generalized Fitting subgroup of $G$. Since $G$ is isomorphic to a subgroup of $\Aut(F)$, it suffices to bound $|F|$.  For this, we want to see that $|S_i|$ and $u_i$ is bounded in terms of $\acd(G)$ for every $i$.

Let $\Delta=\{\theta_1,\dots,\theta_r\}\cup\{\theta_1',\dots,\theta_r'\}\cup\{\theta_{r+1},\dots,\theta_s\}$ be a complete system of representatives of the $G$-orbits on $\Irr(F)$,  where $\theta_i'$ is not defined for $i\in\{r+1,\dots,s\}$. Notice that such a complete system of representatives exists by Lemma \ref{mam}(iv). We have
$$
\Irr(G)=\bigcup_{i=1}^r(\Irr(G|\theta_i)\cup\Irr(G|\theta_i'))\bigcup_{j=r+1}^s\Irr(G|\theta_j)
$$
and these subsets form a partition of $\Irr(G)$. 

Notice that if $\chi\in\Irr(G|\theta_j)$ for some $j>r$ then 
$$
\chi(1)\geq\theta_j(1)\geq \prod_{i=1}^t m(S_i)^{u_i/2},
$$
where we have used the inequality that precedes Lemma \ref{mam}. 
On the other hand the average degree of the irreducible characters of $G$ that lie over $\mathcal{T}_i=\{\theta_i,\theta_i'\}$ (for every $i\leq r$) is
$$
\acd(G|\mathcal{T}_i)\geq\frac{\theta_i(1)+\theta_i'(1)}{2}\geq \frac{1}{2}\prod_{i=1}^t m(S_i)^{u_i/2},
$$
where we have used Lemma \ref{mam}(iii) together with Clifford theory in the first inequality and Lemma \ref{mam}(i) in the second inequality.

Putting everything together, we get 
$$
\acd(G)\geq  \frac{1}{2}\prod_{i=1}^t m(S_i)^{u_i/2}.
$$
If we now take $\acd(G)$ to be a fixed number, we deduce that $t$, $u_i$ and $m(S_i)$ are  bounded in terms of $\acd(G)$. Since $|S_i|$ is bounded in terms of $m(S_i)$ by Lemma \ref{jor}, the result follows.
\end{proof}

Note that our proof of Theorem A depends on the classification of finite simple groups by means of Lemma 4.2 of \cite{mor}. It would be interesting to find a classification-free proof.

\section{Bounding the index of the Fitting subgroup}

In this short section, we prove Theorem B. The following is the large orbit theorem that we will use. In the following, $\alpha=\frac{\log(6\cdot(24)^{1/3})}{\log 3}\approx2.595$. Recall that if a group $G$ acts on a module $V$, $b(G,V)$ is the size of the largest $G$-orbit on $V$.

\begin{thm}
\label{yan}
Suppose that a solvable group $G>1$ acts faithfully and completely reducibly on a module $V$ (of possibly mixed characteristic). Then $|G|<b(G,V)^{\alpha}$.
\end{thm}

\begin{proof}
This is Theorem  3.4 of \cite{yan}.
\end{proof}

Now, we show how to complete the proof of Theorem B. In the following we use the standard reasoning to get bounds for Gluck's conjecture from a large orbit theorem.

\begin{thm}
\label{b}
Let $G$ be a solvable group. Then there exists $\lambda\in\Irr(\bF(G))$ linear such that $|G:\bF(G)|\leq\acd(G|\lambda)^{\alpha}$.
\end{thm}

\begin{proof}
By Gasch\"utz's theorem (Theorem 1.12 of \cite{mw}), $G/\bF(G)$ acts faithfully and completely reducibly on $V=\bF(G)/\Phi(G)$.  By Proposition 12.1 of \cite{mw}, the same holds for the action of $G/\bF(G)$ on $\Irr(V)$. Applying Theorem \ref{yan} to this action, we deduce that there exists $\lambda\in\Irr(V)$ such that the size of the $G/\bF(G)$-orbit of $\lambda$ is 
$$|G:I_G(\lambda)|\geq|G:\bF(G)|^{1/\alpha}.$$
By Clifford's correspondence (Theorem 6.11 of \cite{isa}), all the characters in  $\Irr(G|\lambda)$ are induced from irreducible characters of $I_G(\lambda)$. In particular, if $\chi\in\Irr(G|\lambda)$ then 
$$
\chi(1)\geq|G:I_G(\lambda)|\geq|G:\bF(G)|^{1/\alpha}.$$
It follows that 
$$\acd(G|\lambda)\geq|G:\bF(G)|^{1/\alpha},$$
as desired.
\end{proof}

The proof of the following strong form of Conjecture C for odd order groups is essentially identical, but we use  Theorem 3.2 of \cite{yanodd} rather thanTheorem 3.1.

\begin{thm}
Let $G$ be an odd order group. Then $|G:\bF(G)|\leq\acd(G|\lambda)^{1.643}$.
\end{thm}

As mentioned in the Introduction, all known results on Gluck's conjecture can be quickly adapted to give similar results on Conjecture C. Note also that we could add the condition that $\lambda$ has square-free order, i.e. $\lambda\in\Irr(\bF(G)/\Phi(G))$, in Conjecture C and in all the results in this section.

 \section{Examples}
 
 If $N\trianglelefteq G$, $\Irr(G|N)$ stands for the set of irreducible characters of $G$ whose kernel does not contain $N$. We write $\acd(G|N)$ to denote the average of the degrees of the characters in this set, i.e.,
 $$
 \acd(G|N)=\frac{\sum_{\chi\in\Irr(G|N)}\chi(1)}{|\Irr(G|N)|}.
 $$
 Our first example shows that $|G:\bF(G)|$ is not bounded from above in terms of $\acd(G|\bF(G))$.  This was our first attempt to refine Gluck's conjecture along the lines followed in this paper.
 
 \begin{exa}
 \label{ex1}
 \rm{Let $G=S_3\times F_p$, where $p$ is any odd prime and $F_p$ is the Frobenius group of order $(p-1)p$.  We have that $V=\bF(G)=\CCC_3\times\CCC_p$ and $G/\bF(G)\cong \CCC_2\times\CCC_{p-1}=H$. Note that $G=HV$.  Let $\lambda,\mu\in\Irr(V)$ with $o(\lambda)=3$ and $o(\mu)=p$. The action of $H$ on $\Irr(V)$ has three nontrivial orbits, with representatives $\lambda$, $\mu$ and $\lambda\times \mu$. Their inertia groups in $H$ are, respectively, $\CCC_{p-1}, \CCC_2$ and $1$. Since all of them are cyclic, any character in $\Irr(V)$ extends to its inertia group (by Corollary 11.22 of \cite{isa}). Using Clifford's correspondence and Gallagher's theorem (Corollary 6.17 of \cite{isa}), we deduce that $G$ has $p-1$ irreducible characters of degree $2$ lying over $\lambda$, $2$ irreducible characters of degree $p-1$ lying over $\mu$  and one irreducible character of degree $2(p-1)$ lying over $\lambda\times\mu$. These are all the characters in $\Irr(G|V)$. We deduce that
 $$
 \acd(G|V)=\frac{(p-1)2+2(p-1)+2(p-1)}{p+2}<6,
 $$
 but $|G:\bF(G)|=2(p-1)$ is arbitrarily large. Notice that $\Irr(G|V)$ coincides with the set of nonlinear irreducible characters of $G$.}
 \end{exa}
 
 As we have seen, the proof of Theorem B is a consequence of a large orbit theorem. It is known that Gluck's conjecture cannot be proven as an immediate consequence of a large orbit theorem. More precisely, there are examples of (even order) solvable groups $G$ acting faithfully and completely reducibly on finite modules $V$ without orbits of size at least $|G|^{1/2}$ (see Example 13 of \cite{wol}, for instance). Any counterexample to Gluck's conjecture, or to Conjecture C, should involve these actions. We show that the group in Example 13 of \cite{wol} is not a counterexample to Conjecture C.
 
 \begin{exa}
 \rm{
 Let $G=\SSS_4\wr\SSS_3$.  Recall that  $\SSS_4$ is the semidirect product of $\SSS_3$ acting faithfully and irreducibly on $W=C_2\times C_2$, so that $V=\bF(G)$ is elementary abelian of order $2^6$. Write $G=HV$ for some subgroup $H\cong G/\bF(G)$. As in Example 13 of \cite{wol}, the $G$-orbits in $\Irr(V)$ have size $1, 9, 27$ and $27$. Let $\mu\in\Irr(W)$ be nonprincipal so that $\lambda=\mu\times\mu\times\mu\in\Irr(V)$ lies in a $G$-orbit of size $27$. Note that $I_H(\lambda)=\CCC_2\wr\SSS_3$ and $I_G(\lambda)=\QQQ_8\wr\SSS_3$. By Problem 6.18 of \cite{isa},  $\lambda$ extends to its inertia subgroup in $G$. We  leave as an exercise to check that $I_H(\lambda)=I_G(\lambda)/V$ has $4$ linear characters,, $2$ irreducible characters of degree $2$ and $4$ irreducible characters of degree $3$. In particular, its average character degree is $2$. Using Gallagher's theorem, we conclude that  
$\acd(G|\lambda)=2\cdot27=2\cdot3^3$. Thus
$$|G:\bF(G)|=2^4\cdot3^4<2^2\cdot3^6=\acd(G|\lambda)^2,$$
as desired.

On the other hand, if $\varphi_1,\varphi_2$ and $\varphi_3$ are the three nonlinear irreducible characters of $\SSS_4$ (two of them of degree $3$ and the other one of degree $2$), then $\chi=(\varphi_1\times\varphi_2\times \varphi_3)^G\in\Irr(G)$. We have that $\chi(1)=108=b(G)$, while $\acd(G|\lambda)=54$. This example shows that even in the case when Gluck's conjecture does not follow from a large orbit theorem, there is perhaps room for improvement in the bound predicted by Gluck's conjecture.}
\end{exa}

 \section{Further remarks and questions}
 
 We start with a very fundamental question on the average character degree. It is even surprising that some results on $\acd(G)$ have been obtained without an answer to it.
 As mentioned in the Introduction, an affirmative answer to this question, together with Theorem A,  would imply that $|G:\sol(G)|$ is bounded from above in terms of $\acd(G)$. 
 
 \begin{que}
 Let $G$ be a finite group and $N\trianglelefteq G$. Is it true that $\acd(G/N)$ is bounded from above in terms of $\acd(G)$? Is it true that even $\acd(G/N)\leq\acd(G)$?
 \end{que}

Recall that it was proved in Theorem 9 of \cite{gr} that  $|G:\sol(G)|\leq\acs(G)^2$. We propose the following question.

\begin{que}
\label{1}
Let $G$ be a finite group. Is it true that $|G:\sol(G)|\leq\acd(G)^4$?
\end{que}

An interesting first step would be to achieve this bound for groups with trivial solvable radical. Using results from \cite{chmn}, it has been shown by N. N. Hung that this bound holds when $G$ is simple.

Gluck's conjecture was extended to arbitrary finite groups in Question 5 of \cite{chmn}, where it was asked whether $|G:\bF(G)|\leq b(G)^3$ for any finite group.  It is interesting to note that as a consequence of the Guralnick-Robinson inequality $|G:\bF(G)|\leq\acs(G)^2$, it was proved in \cite{chmn} that  $|G:\bF(G)|\leq b(G)^4$. As $\AAA_5$ shows, it is not true that if $G$ is a finite group then $|G:\sol(G)|\leq\acd(G)^3$. 
We do not know any counterexamples to the following question for arbitrary finite groups.

\begin{que}
\label{2}
Let $G$ be a finite group. Is it true that there exists $\lambda\in\Irr(\bF(G))$ linear such that $|G:\bF(G)|\leq\acd(G|\lambda)^4$?
\end{que}

We have seen in Example \ref{ex1} that, unlike for class sizes,  the index of the Fitting subgroup of a solvable group cannot in general be bounded in terms of $\acd(G)$. We have been unable to decide the answer to the following question, even in the $p$-group case (see Theorem 12 of \cite{gr} for the corresponding result for class sizes).

\begin{que}
Let $G$ be a solvable group. Is it true that the derived length of $G$ is bounded in terms of $\acd(G)$?
\end{que}

\end{document}